\documentclass[reqno]{amsart}
\usepackage{latexsym, amsmath, amscd, amssymb, amsthm, bm, mathrsfs}
\usepackage{amsrefs}
  \usepackage{multicol}
  \usepackage{graphicx} 
  \usepackage{xcolor}
  \usepackage[colorlinks=true,linkcolor=blue,urlcolor=blue,
  citecolor=blue]{hyperref}
  
  \newtheorem{lemma}{Lemma}[section]
\newtheorem{thm}[lemma]{Theorem}
\newtheorem{theorem}[lemma]{Theorem}
\newtheorem{cor}[lemma]{Corollary}
\newtheorem{prop}[lemma]{Proposition}

\theoremstyle{definition}
\newtheorem{definition}[lemma]{Definition}

\newtheorem{example}[lemma]{Example}
\theoremstyle{remark}

\DeclareMathOperator{\Conv}{Conv}
\DeclareMathOperator{\Eu}{Eu}
\DeclareMathOperator{\Vol}{Vol}
\DeclareMathOperator{\rowspan}{rowspan}
\DeclareMathOperator{\rank}{rank}

\newcommand{\mb}{\mathbf}

\newcommand{\Z}{\mathbb{Z}}
\newcommand{\Q}{\mathbb{Q}}
\newcommand{\R}{\mathbb{R}}
\newcommand{\C}{\mathbb{C}}
\newcommand{\N}{\mathbb{N}}

\begin{document}

\title{Interpolation of toric varieties}
\author[A.~Dickenstein, S. ~Di Rocco, R.~Piene]{Alicia Dickenstein, Sandra Di Rocco and Ragni Piene}
\address{Departamento de Matem\'atica, FCEN, Universidad de Buenos Aires, Ciudad
Universitaria - Pab. I, (1428) Buenos Aires, Argentina}
\email{\href{mailto:alidick@dm.uba.ar}{alidick@dm.uba.ar}}
\urladdr{\href{http://mate.dm.uba.ar/~alidick}
{http://mate.dm.uba.ar/~alidick}}
\address{Sandra Di Rocco \\ Department of Mathematics, KTH Royal Institute of Technology\\ SE 10044 Stockholm, Sweden}
\email{\href{mailto:dirocco@kth.se}{dirocco@kth.se}}
\urladdr{\href{https://www.kth.se/profile/dirocco}
{https://www.kth.se/profile/dirocco}}
\address{Ragni Piene\\Department of Mathematics\\
University of Oslo\\P.O.Box 1053 Blindern\\NO-0316 Oslo\\Norway}
\email{\href{mailto:ragnip@math.uio.no}{ragnip@math.uio.no}}
\urladdr{\href{https://www.mn.uio.no/math/english/people/aca/ragnip/index.html}
{https://www.mn.uio.no/math/english/people/aca/ragnip/index.html}}

\date{\today}

\begin{abstract}
Let $X\subset \mathbb P^d$ be an $m$-dimensional  variety in $d$-dimensional complex projective space. Let $k$ be a positive integer 
such that $\binom{m+k}k \le d$. Consider the following \emph{interpolation problem}: does there exist a variety 
$Y\subset \mathbb P^d$ of dimension $\le \binom{m+k}k -1$, with $X\subset Y$, such that the tangent space to $Y$ at a point $p\in X$
 is equal to the $k$th osculating space to $X$ at $p$, for almost all points $p\in X$?
In this paper we consider this question in the \emph{toric} setting. We prove that if $X$ is toric, then there is a unique toric variety $Y$ solving the above interpolation problem. We
identify $Y$ in the general case and we explicitly compute some of its invariants when $X$ is a toric curve.
\end{abstract}

\maketitle

\section{Introduction}

When a problem can be modeled by  polynomial equations, where their solutions correspond to an algebraic variety, it often becomes important to determine the optimal variety that satisfies given constraints. 
The classical \emph{interpolation problem} in algebraic geometry is the following:  find all plane curves of a prescribed degree passing through a given set of points.
More generally, one can consider a class of varieties together with a collection of incidence conditions, involving linear subspaces and possibly tangency or higher order osculating conditions, and ask for those varieties in the class that satisfy the given conditions.
This is both a theoretically and computationally challenging problem in algebraic geometry and related fields.

A classical example in real differential geometry is the following \cite{BL}*{p.~56}: Consider a space curve $C$. Find a space curve such that its osculating planes are equal to the normal planes of $C$. The solution to this question is the \emph{evolute} of the curve $C$, namely the locus of its spherical curvature centers. The study of osculating spaces to a space curve goes back at least to Monge and others in the 18th century. For the case of curves in higher dimensional spaces, see \cite{numchar}. 

The $k$th \emph{osculating space} to a variety $X\subset\mathbb P^d$   at a point  $p\in X$ is a linear space of dimension $\le \binom{m+k}k -1$ in $\mathbb P^d$ that is tangent to $X$ at $p$ to the order $k$. The osculating space at $p\in X$ of order $k=1$ is the embedded tangent space, denoted by $T_{X,p}.$ 
Osculating spaces have been extensively studied in the context of higher order dual varieties, see \cites{note,higher} for the general case and \cites{ASR,AR} for the case of toric varieties. 
We refer to \ref{ssec:osculating} for more details.

 The kind of interpolation problem we consider in this article is the following. Fix a variety $X\subset \mathbb P^d$ of dimension $m$  in complex projective space of dimension $d$. Let $k$ be a positive integer satisfying $\binom{m+k}k \le d$. Consider the set of all varieties $Y\subset \mathbb P^d$ of dimension  $\le \binom{m+k}k -1$ such that $X\subset Y$. We say that  $Y$ satisfies the $k$th \emph{interpolation condition} with respect to $X$ -- or that $Y$ is a \emph{$k$th interpolant of $X$} -- if the embedded tangent space to $Y$ at
  almost all points of $X$ is equal to the $k$th osculating space to $X$ at that point. 
A natural question 
is then:

 \emph{ Determine the existence and uniqueness of a $k$th interpolant, and, if it exists, explore methods for its construction.}
 
  In general, answers are expected to be challenging, particularly with regard to uniqueness, as we are seeking a unique object that satisfies specific local conditions. However, in the case of toric varieties, the rigidity imposed by the torus action on the geometry allows us to provide a complete solution to the problem. 
  
 If $X$ and $Y$ are toric varieties and $Y$ satisfies the $k$th interpolation condition with respect to $X$, we call $Y$ a $k$th \emph{toric interpolant} of $X$.

A toric variety  is a (not necessarily normal) 
 algebraic variety  containing an algebraic torus as a Zariski open set 
 and such that the multiplicative self-action of the torus extends to the whole variety. Projective spaces are toric varieties and the torus of $\mathbb P^d$ is the open subset $T^d$ of projective points with all nonzero coordinates. The action of  $T^d$ on $\mathbb P^d$ is given by coordinatewise multiplication, that is,  multiplication by a point $(t_0:\dots : t_d) \in T^d$ is given by the diagonal matrix with diagonal entries $(t_0, \dots, t_d)$. 
 Equivariantly embedded projective toric varieties $X \subset \mathbb P^d$ of dimension $m$ not contained in a coordinate hyperplane are rational varieties parameterized by monomials
 with exponents in the columns of a matrix $A \in \Z^{(m+1) \times (d+1)}$ of rank $m$, which give the weights of the torus action in~\eqref{eq:action} below~\cite{GKZ}*{Ch.~5, Prop.~1.5}.  In this case, we denote as usual $X=X_A$. This configuration of column vectors lies on a hyperplane off the origin and so  we will assume without loss of generality that the first coordinate of each of these column vectors is equal to $1$.  In fact $X_A$ is associated to the affine equivalence class of $A$ by~\cite{GKZ}*{Ch.~5, Prop. 1.2}. We refer the reader to Section~\ref{toricinterp}, where we recall with more detail this characterization and we show in Theorem~\ref{thm:unique} that the variety $X_A$ is completely determined by its embedded tangent space $T_{X_A, (1 : \dots :1)}$.
 
Given $A \in \Z^{(m+1) \times{(d+1)}}$ and a positive integer $k$, consider the matrix  $A^{(k)}$ in $\Z^{\binom{m+k}{k} \times (d+1)}$ 
given in Definition~\ref{def:Ak}, originally introduced in~\cite{ASR} in connection with the  study of higher order dual varieities of toric varieties. In Theorem~\ref{th:interpolant} we prove existence and unicity of toric interpolants, while providing an explicit construction:
\[Y=X_{A^{(k)}}\subset \mathbb P^d \text{ \emph{is the unique $k$th toric interpolant  of} } X_A.\]

We explain that, thanks to the torus action, 
this happens as soon as the tangent space to $Y$  at one point $p$ with all nonzero coordinates equals the $k$th osculating space to $X_A$ at $p$.

In Section \ref{sec:curves} we analyze in more details the case of toric curves.
Let us first recall the concept of cyclic polytopes associated to curves.  The $m$-moment curve in $\mathbb {R} ^{m+1}$ is defined by the image of the map $\alpha_m: \mathbb R \to \mathbb R^{m+1}$ sending $t$ to the vector $(1,t, t^2,\dots, t^m)$. An $m$-dimensional \emph{cyclic polytope} is defined as the convex hull of the image, under the moment curve, of a finite number (at least two), of ordered distinct points. It is known that all the the images of these points by $\alpha_m$ are vertices of the cyclic polytope, which has dimension $m$. Its combinatorial structure is independent of the points chosen. 
Toric curves are associated to matrices $A \in \Z^{2 \times (d+1)}$ of the form
\[\begin{pmatrix}
1  & 1 & \dots & 1 & 1\\
\ell_0 & \ell_1 &\dots & \ell_{d-1} & \ell_d
\end{pmatrix},\]
where we will always assume, without loss of generality, that $\ell_0 < \ell_1 < \dots < \ell_d$. 
We describe the degree and number of lattice points of the second interpolants of toric curves.
  All matrices $A^{(k)}$ arising from toric curves are {\em positroids} (i.e., all their maximal minors are non-negative) and thus their
convex hulls 
define {\em positive geometries} in the sense of ~\cite{PosGeo}. The columns of the matrices $A^{(k)}$ correspond 
to the vertices of a cyclic lattice polytope of dimension $k$.  In the case $k=2$
we consider the associated polygon  and we compute its canonical form. This is an example of a generalized \emph{tree amplituhedron}  of type $\mathcal A_{d+1,1,2}(Z)$ \cites{PosGeo,ampli} in the 
positive Grassmann variety $\mathbb G^{\ge 0}_{1,3}=(\mathbb P^2)^{\ge 0}$.  Indeed,  amplituhedra are geometric objects studied  in particle physics in relation to integral representations of scattering amplitudes and they are interesting mathematical objects that naturally generalize cyclic polytopes. 
Our computations can be extended to $k$th interpolants for any $k$, based on the description of the facets of combinatorial cyclic polytopes~\cite{CD}.
In the standard case of the toric curves with $\ell_i=i$ for any $i=0, \dots, d$,  we compute the degree of the second and the $(d-1)$th interpolant.

In Section~\ref{sec:dual} we describe the normalization of the second interpolant of a toric curve. By computing the degree of its dual variety and the degree of the dual variety of its normalization, we conclude that though the second interpolant is a linear projection of its normalization, it is not a \emph{general} linear projection.

\section*{Acknowledgements} This work was partially supported by the project Pure Mathematics in Norway, funded by Trond Mohn Foundation and 
Tromsø Research Foundation. AD was partially supported by UBACYT 20020220200166BA and CONICET PIP 20110100580, Argentina.  
We would like to thank the referees for their vaulable comments and questions that have helped us improve the exposition.

\section{Toric interpolation}\label{toricinterp}
In this section we recall some known facts about projective toric varieties $X_A$ associated to a matrix $A$ (\ref{ssec:toricgeo}) and the notion of higher osculating spaces of a
projective variety (\ref{ssec:osculating}). For toric varieties  we prove that hyperosculation of order $k$ can be characterized by an associated matrix $A^{(k)}$ naturally built from $A.$
 In Definition \ref{def:interpolant} we introduce the notion of $k$th (toric) interpolant and in Theorem~\ref{th:interpolant} we prove that the $k$th toric interpolant of the equivariantly embedded projective toric variety $X_A$ is unique and is equal to the toric variety $X_{A^{(k)}}$. Uniqueness is based on the basic result in Theorem~\ref{thm:unique}.

\subsection{Projective toric varieties}\label{ssec:toricgeo}

Let $A=\{(1, {\mathbf a_0}),\ldots,(1, {\mathbf a_d})\}\subset\mathbb Z^{m+1}$ be a  finite set of lattice points in an affine hyperplane off the origin. We denote also by $A$ the
$(m+1)\times(d+1)$-matrix whose columns are given by the $(1, {\mathbf a_j})$'s. Consider the map  
\begin{equation}\label{eq:iota}
\iota_A:(\mathbb C^*)^{m}  \to  \mathbb P^d \quad \text{ defined by } \quad
\mathbf t:=(t_1,\ldots, t_m) \mapsto (\mathbf t^{\mathbf a_0}:\cdots : \mathbf t^{\mathbf a_d}),
\end{equation}
 where
$\mathbf t^{\mathbf a_j}=\prod_i t_i^{a_{i,j}}$. 
Note that if we consider instead the map
\begin{equation}\label{eq:jota}
\iota'_A:(\mathbb C^*)^{m+1}  \to  \mathbb P^d \quad \text{ defined by } \quad
\mathbf (t_0,\mathbf t) \mapsto (t_0 \mathbf t^{\mathbf a_0}:\cdots : t_0 \mathbf t^{\mathbf a_d}),
\end{equation}
we have that $\iota_A(\mathbf t)= 
\iota'_A (t_0, \mathbf t)$.

The projectively embedded toric variety $X_A \subset \mathbb P^d$ associated with $A$ is defined to be the Zariski closure  of the image of $\iota_{A}$. This image
is the torus $X_A \cap \{ x \in \mathbb P^d \, : \, x_i \neq 0, \, i =0, \dots, d\}$ of $X_A$ and it always contains the point \[{\mathbf 1}:=(1 \, : \, \ldots \, : \, 1) \in 
 \mathbb P^d.\]
 
Also, $X_A=\overline{{\rm Orb}({\mb 1})}$ is the closure of the orbit of the
point ${\mb 1}$ by the diagonal action
\begin{equation}\label{eq:action}
{\mb t} {*}_A (x_0: \dots : x_d) \, = \, ({\mb t}^{\mb a_0} x_0: \dots : {\mb t}^{\mb{a_d}} x_d).
\end{equation}

\smallskip

The variety $X_A$  is an affine invariant of the configuration $A$ and its dimension  equals the affine dimension of 
$A$~\cite{GKZ}*{Prop.~1.2, Ch.~5}. As we mentioned in the introduction, we will assume without loss of generality that all points in $A$ have first coordinate equal to $1$, which implies that $\dim(A)=\rank(A)-1$.
We will moreover assume without loss of generality that the matrix $A$ has maximal rank $m+1$, or equivalently, that the convex hull of the points ${\mathbf a_0}, \dots, {\mathbf a_d}$ is of maximal dimension $m$. As we also mentioned, Proposition~1.5 in Chapter~5 of the book~\cite{GKZ} by Gelfand, Kapranov and Zelevinsky shows that any projective toric variety with an equivariant embedding (that is, with a diagonal torus action) and not contained in a coordinate hyperplane, is of the form $X_A$.  The degree of the projective variety $X_A$ equals $\Vol(A),$ the normalized volume of $A$ (cf. Theorem~4.16 in~\cite{BS}). Subtracting multiples of the first row from the other rows, we can also assume without loss of generality that ${\mb a_0} = {\mb 0}$.  When the subgroup $\Z A$ generated by ${\bf a_0}=0, \dots, {\bf a_d}$ equals $\Z^m$,
the normalized volume  $\Vol(A)$ of $A$ is defined as $m!$ times the Euclidean volume of the convex hull of the lattice configuration $A$. 
Otherwise, it equals this quantity divided by the rank of the quotient $\Z^m/ \Z A$.
\smallskip

We present a very simple example:

\begin{example} \label{ex:simple}
Consider the matrices
\[
A_1 = \begin{pmatrix}
1 & 1 & 1 & 1 \\
0 & 1 & 2 & 3
\end{pmatrix},
\quad
A_2 = \begin{pmatrix}
3 & 2 & 1 & 0 \\
0 & 1 & 2 & 3
\end{pmatrix},
\quad
A_3 = 
\begin{pmatrix}
1 & 1 & 1 & 1 \\
0 & 1 & 2 & 3 \\
1 & 2 & 3 & 4
\end{pmatrix},
\]
and $A_4 = \begin{pmatrix}
0 & 1 & 2 & 3.
\end{pmatrix}$. 
The closure of the images by the corresponding maps  $\iota_{A_j}, \, j=1,3,4$,   in~\eqref{eq:iota}, equals the same {\em projective} toric variety: the rational normal curve of degree $3$ in $\mathbb P^3$ cut out by the following equations:
\[ 
\{ (x_0: x_1: x_2 :x_3) \in \mathbb P^3\, | \, x_1^2 - x_0 x_2 =0, x_2^2 -x_1 x_3=0, x_0 x_3 - x_1 x_2 =0 \}.\]
See also Example~\ref{ex:q=2} below.

Note that the columns of $A_1$ are the injective image of the columns of $A_4$ by the affine map $m \to (1,m)$. The matrix $A_3$ has rank $2$ and its columns are the injective image of the columns of $A$ by the linear map $(m_1, m_2) \mapsto (m_1, m_2, m_1+m_2)$.

The columns of the matrix $A_2$ equal the image of the columns of $A_1$ via the linear map $(m_1, m_2) \mapsto (3 m_1 - m_2, m_2)$ and $A_2 = M A_1$, where $M \in GL(2, \mathbb Q)$ is the matrix
\[
M = \begin{pmatrix}
    3 & -1 \\
    0 &  1
\end{pmatrix}.
\]
Indeed, the first row of $A_2$ is not the all $1$ vector, but  $(1, \ldots, 1)$ is in the rowspan of the matrix $A_2$. This is clear since $M^{-1} A_2 = A_1$,   It is easy to check, as we remarked in~\eqref{eq:jota}, that the map $j_{A_2}$ defined by $j_{A_2}(t_0, t_1) =(t_0^3: t_0^2 t_1: t_0 t_1^2: t_1^3)$ verifies
$j_{A_2}(t_0, t_1) = (1: s: s^2: s^3) = \iota_{A_1}(s)$ for $s = t_1/t_0$. Thus $A_2$ also gives a rational parametrization of the rational normal curve of degree $3$.

The vectors $(1,-2,1,0), (0, 1,-2,1), (1,-1,-1,1)$ that we can read from the
exponents of the equations generate the kernel of the matrices $A_1, A_2, A_3$, and the space of affine relations among the columns of $A_1$ (that is, elements of the kernel that add up to $0$). Note that it is not enough to select a basis of the kernel; for instance if we omit the equation $x_0 x_3 - x_1 x_2 =0$, the points in the line $\{x_1=x_2=0\}$ also satisfy the first two equations.  Instead, any two of the equations describe the variety outside the coordinate hyperplanes.
\end{example}

\subsection{Osculating spaces} \label{ssec:osculating}
Let $X\subset \mathbb P^d$ be a projective algebraic
variety of dimension $m$. Consider the sheaf  ${\mathcal L}:={\mathcal
O}_{\mathbb P^d}(1)|_X$, and let ${\mathcal P}_{X}^k({\mathcal L})$ denote the sheaf of $k$th order principal parts of 
$\mathcal L$ \cite{numchar}*{\S~6, p.~492}. Recall that the rank of  ${\mathcal P}_{X}^k({\mathcal L})$ at a smooth point $p \in X$ is $\binom{m+k}{k}$.
Indeed, the fiber ${\mathcal P}_{X}^k({\mathcal L})_x$ at a point
$p\in X$ is isomorphic to the vector space $\mathcal O_{X,x}/\mathfrak m^{k+1}_{X,p}$, where $\mathfrak m_{X,p}\subset \mathcal O_{X,p}$ is the maximal ideal in the local ring of $X$ at $p$. Principal parts bundles play a crucial role in the study of projective duality and differential properties of projective embeddings, see \cite{K}*{ IV.A., pp.~341--346; IV.D., pp.~359--365}.

Assume that $U\subset X$ has a parameterization
\[\mathbf t:=(t_1,t_2,\dots,t_m)\mapsto (x_0(\mathbf t):\cdots:x_d(\mathbf t)).\]
Then the $k$th jet map
\[j_k: 
{\mathcal O}_X^{d+1}\to  {\mathcal P}_{X}^k({\mathcal L})\]
restricted to $U$ is given by the matrix
\begin{equation}\label{Ak}
A^{(k)}(\mathbf t):=\left(
\begin{array}{cccc} 
x_0(\mathbf t)&x_1(\mathbf t)&\cdots &x_d(\mathbf t)\\
\frac{\partial x_0(\mathbf t)}{\partial t_1}&\frac{\partial x_1(\mathbf t)}{\partial t_1}&\cdots&\frac{\partial x_d(\mathbf t)}{\partial t_1}\\
\vdots & \vdots & \vdots & \vdots\\
\frac{\partial x_0(\mathbf t)}{\partial t_m}&\frac{\partial x_1(\mathbf t)}{\partial t_m}&\cdots&\frac{\partial x_d(\mathbf t)}{\partial t_m}\\
\frac{\partial^2 x_0(\mathbf t)}{\partial t_1^2}&\frac{\partial^2 x_1(\mathbf t)}{\partial t_1^2}&\cdots&\frac{\partial^2 x_d(\mathbf t)}{\partial t_1^2}\\
\frac{\partial^2 x_0(\mathbf t)}{\partial t_1 \partial t_2}&\frac{\partial^2 x_1(\mathbf t)}{\partial t_1 \partial t_2}&\cdots&\frac{\partial^2 x_d(\mathbf t)}{\partial t_1 \partial t_2}\\
\vdots & \vdots & \vdots & \vdots\\
\frac{\partial^k x_0(\mathbf t)}{\partial t_m^k}&\frac{\partial^k x_1(\mathbf t)}{\partial t_m^k}&\cdots&\frac{\partial^k x_d(\mathbf t)}{\partial t_m^k}\
\end{array}\right),
\end{equation}
where $\frac{\partial^j}{\partial t_i^j}$ denotes the Hasse derivative, that is, $\frac{1}{j!}$ times the standard derivative. 

Given a matrix $A$, we will denote by $\rowspan(A)$ the subspace of $\mathbb C^{d+1}$ spanned by the row vectors of $A$.
The \emph{$k$th osculating space} to $X$ at a point corresponding to ${\mb t}$ is  the linear space 
\[\mathbb P (\rowspan(A^{(k)}(\mb t))).\]

\subsection{The matrices $A^{(k)}$}\label{ssec:Ak}

We now construct matrices $A^{(k)}$ 
describing the higher osculating spaces of a toric variety $X_A$ at ${\mb 1} = (1: \dots: 1)$. 
For any $k$, the $k$th osculating spaces at the points $\iota_A({\mb t}) = {\mb t} {*}_A {\mb 1}$ in 
the torus of $X_A$ are translated by this action (defined in~\eqref{eq:action}).

We give two different constructions of matrices. They will define the same projective toric variety.

\begin{definition} \label{def:Ak}
Let $A \in \Z^{(m+1) \times (d+1)}$ with column vectors $(1, {\mb a_j}), j = 0, \dots d$, and $k \in \N$.
We define the associated matrix $\widetilde A^{(k)}$ as follows.
We will add $\binom{m+k}k -(m+1)$ new rows to $A$ in its lower part. We order the vectors $\{\mathbf i =(i_1,\dots,i_m) \in \Z^m_{\ge 0}\, | \, 2\le |\mathbf i|\le k\}$, with lexicographic order
with $0 < 1 < \dots < m$ and use them to label these rows.
 The entry in the matrix $\widetilde A^{(k)}$ in the $j$th column and the $\mathbf i$th row is the integer $a_{1,j}^{i_1}\cdots a_{m,j}^{i_m}$.
 When all the coefficients of $A$ are nonnegative,
we define another associated matrix $A^{(k)}$ as follows. 
Again, we add $\binom{m+k}k -(m+1)$ new (ordered) rows to $A$ in its lower part labeled by the vectors $\mathbf i =(i_1,\dots,i_m)$.
 The entry in the matrix $A^{(k)}$ in the $j$th column and the $\mathbf i$th row is the integer $\binom{\mathbf a_j}{\mathbf i}:=\binom{a_{1,j}}{i_1}\cdots \binom{a_{m,j}}{i_m}$.
\end{definition}

Since $\binom{a}{i} = \frac{a (a-1) \dots (a-i+1)}{i!}$ for any positive integer $a \ge i$, it is straightforward to see that  for any $k$ there exists an integer matrix $M_k$ such that $\widetilde A^{(k)} = M_k \cdot A^{(k)}$ with  $\det(M_k)= \prod_{2 \le |{\bf i}| \le k} i_1!\dots i_m!$. 
By \cite{ASR}*{2.2}, the projective linear space 
\[\mathbb P (\rowspan (A^{(k)})) =\mathbb P (\rowspan (\widetilde A^{(k)}))\]
equals the
$k$th osculating space of $X_A$ at the point ${\mathbf 1}$. This is easily seen by considering the parameterization $\iota_A$ to construct the matrix $A^{(k)}({\bf t})$  in (\ref{Ak}).  Evaluating at ${\bf t}={\bf 1}$,  we get the matrix $A^{(k)}$.

\begin{example} \label{ex:q=2}
 Let $d \in \mathbb N$ and let $A$ be the matrix
 \[
A_d=\left(
\begin{array}{cccccc} 1&1&1&1& \cdots &1\\
0&1&2&3& \cdots & d
\end{array}\right).
 \] 
When $k=2$,  we have
\[
\begin{array}{cc}
A_d^{(2)}=\left(
\begin{array}{cccccc} 1&1&1&1& \cdots &1\\
0&1&2&3& \cdots & d\\
0&0& 1&3& \cdots &\frac{d(d-1)}{2}
\end{array}\right), \, & \, 
\widetilde A_d^{(2)}=\left(
\begin{array}{cccccc} 1&1&1&1& \cdots &1\\
0&1&2&3& \cdots & d\\
0& 1& 4&9& \cdots & d^2
\end{array}\right),
\end{array}\]
and $\widetilde A_d^{(2)} \, = \, M_2  \cdot A_d^{(2)}$, 
where $M_2 \in {\rm {GL}(3,\Q})$
is the matrix
\[
M_2=\left(
\begin{array}{ccc} 
1&0&0\\
0&1&0\\
0&1& 2
\end{array}\right).
\]
 Consider the rational parameterizations of $X_{A_d^{(2)}} = X_{\widetilde A_d^{(2)}}$   from $ (\C^*)^2$ to $\mathbb P^d:$ 
\[ \iota_{A_d^{(2)}}:t \mapsto (1 : t_1 : t_1^2 t_2:\dots : t_1^d t_2^{\frac{d (d-1)}{2}}), \quad \iota_{\widetilde A_d^{(2)}}: (s_1,s_2) \mapsto (1 : s_1 s_2: s_1^2 s_2^4: \dots: s_1^d s_2^{d^2}),\] 
and let $\varphi_{M_2}: (\C^*)^2 \to (\C^*)^2$ be the  $2:1$-map $s \mapsto (s_1 s_2, s_2^2)$. Then,
it is straightforward to check that $\iota'_{\widetilde A_d^{(2)}} = \iota_{A_d^{(2)}} \circ \varphi_{M_2}$.
We could also consider the first rows of the matrices and get a parameterization of the same (projective) variety from $(\C^*)^3$ to $\mathbb P^d$, given by the map $\iota'_A(t_0, t_1, t_2) = (t_0: t_0 t_1 : \dots : t_0 t_1^d t_2^{\frac{d (d-1)}{2}})$. 

As noted above, the tangent space at $\mb 1$ of the variety $X_{A_d^{(2)}} = X_{\widetilde A_d^{(2)}}$
 equals the second osculating space of $X_A$ at $\mb 1$.
The standard cyclic polygon of $d+1$ points in the moment curve in the plane is in general presented with vertices in the columns of the
matrix $\widetilde A_d^{(2)}$, but  $\iota'_{A_d^{(2)}}$ is $1:1$ while $\iota'_{\widetilde A_d^{(2)}}$ is $2:1$. 
\end{example}

\subsection{The toric interpolant}\label{ssec:interpolant}
We define $k$th toric interpolants for any $k \in \mathbb N$ and we  show in Theorem~\ref{th:interpolant} that a toric interpolant always exists and that it is unique. 

\begin{definition}\label{def:interpolant} Let $X_A\subset \mathbb P^d$ be a projective toric variety and  $k \ge 1$ a natural number. 
\begin{enumerate}
    \item We say that a projective toric variety  
$X_B$ is a $k$th toric interpolant  of $X_A$ at $p\in X_A$ if $X_A \subset X_B$ and the tangent space to $X_B$ at the point ${\bf p} \in X_A$
 is equal to the $k$th osculating space to $X_A$ at $\bf p$. 
 \item We say that $X_B$ is a $k$th toric interpolant  of $X_A$ if $X_B$ is a $k$th toric interpolant  of $X_A$ at almost all points ${\bf p}\in X_A$. 
\end{enumerate}
 \end{definition}

We start by showing the basic result that a projective toric embedding  $X_A\hookrightarrow \mathbb P^d$ is completely determined by the tangent space at the point $\mathbf 1  \in X_A\subset\mathbb P^d$.

\begin{theorem}
    \label{thm:unique}
    Let $A\in \Z^{(m+1) \times (d+1)}$ be such that $\rowspan(A)\subset\R^{d+1}$ is a subspace
    of dimension $m+1$ containing the vector $(1,\dots, 1) \in \R^{d+1}$.
    Then, the following statements hold:
    \begin{itemize}
    \item[(i)] Given another matrix $A'\in \Z^{(m+1) \times (d+1)}$  such that $\rowspan(A')$ has dimension $m+1$ and contains the vector $(1,\dots, 1) \in  \R^{d+1}$, then
        $X_A = X_{A'}$ if and only if
    $\rowspan(A)$ = $\rowspan(A')$.
    
    \smallskip
    
    \item[(ii)] The embedded tangent space to $X_A$ at $\mathbf 1 \in X_A\subset \mathbb P^d $ is the projectivization   $\mathbb P({\rm row span}(A) \otimes_{\mathbb R} \mathbb C)$.
    \end{itemize}

    \smallskip

Moreover, let $L$ be a linear subspace in $\R^{d+1}$ of dimension $m+1$  which is defined over $\mathbb Q$  such that
$(1, \dots,1) \in L $.
Consider the projectivization $\mathbb P(L_{\mathbb C})$ of its extension $ L_{\mathbb C} = L\otimes_{\mathbb R} \mathbb C$. There exists a unique equivariantly embedded toric variety $X_A$ such that the image $\mathbb P(L_{\mathbb C})$ in projective space of $L_{\mathbb C}$ is the embedded tangent space $T_{X_A, \mb 1}$ of $X_A$ at $\mathbf 1$. Indeed, it is enough to take any matrix $A \in \Z^{(m+1) \times (d+1)}$ such that $L = {\rm row span}(A)$.
  \end{theorem}

\begin{proof}
If $\rowspan(A)$ = $\rowspan(A')$ we have that $X_A=X_{A'}$ because the configurations of columns of $A$ and $A'$ are affinely equivalent.

Let $\iota_A$ be the rational parameterization of $X_A$ defined in~\eqref{eq:iota} and 
$U= \iota_A((\mathbb C^*)^m)$ the torus of $X_A$. Taking $k=1$ in~\eqref{Ak} we get that  $A^{(1)}(1, \ldots, 1) = A$.
Therefore, we see that the embedded tangent space at the point $\mathbf 1$ is the projective linear space spanned by the row vectors of $A$ considered as points in $\mathbb P^d$, which shows item (ii).
Then, if $X_A = X_{A'}$, they have the same embedded tangent space at $\mb 1$ and so $\rowspan (A) = \rowspan(A') $

Given such a linear subspace $L$ as in the statement and any choice of matrix $A$ with $L = \rowspan(A)$, 
then $T_{X_A, \mb 1} =\mathbb P (\rowspan(A) \otimes_{\mathbb R} \mathbb C) =  \mathbb P(L_{\mathbb C})$.  \end{proof}

We next show the existence and uniqueness of toric interpolants for any $k \in \mathbb N$.
 
 \begin{thm}\label{th:interpolant}
 Let $A \in \Z^{(m+1) \times (d+1)}$ and $k \in \N$.
 \begin{itemize}
 \item[(i)] The variety $X_{A^{(k)}}$ is a $k$th toric interpolant of $X_A$ at {\em all points in the torus}  of $X_A$, i.e., all points of $X_A$ in the torus of $\mathbb P^d$.
\item[(ii)]  Assume  there exists a matrix $B \in \Z^{(m_B+1) \times (d+1)}$ such that $X_B$ is a 
 $k$th toric interpolant of $X_A$ at one point ${\bf p}^* = \iota_A({{\bf t}^*})$ in the torus of $X_A$. Then $X_B = X_{A^{(k)}}$.
 
 \end{itemize}

 \end{thm}

\begin{proof}
Assume that the first row of $A$ is given by the vector with all coordinates equal to $1$. We have already observed that  for any ${\bf t} \in (\C^*)^m$, the tangent space to $X_A$ at
${\bf p}= \iota_A  ({\bf t}) = (\mathbf t^{\mathbf a_0}:\cdots : \mathbf t^{\mathbf a_d})$ is spanned by the row vectors of the matrix $A^{(1)}(\mathbf t)$ in (\ref{Ak}). 
Note that multiplying the $j$-th row vector of this matrix by $t_j$ for $j=1, \dots, m$, we get that the rows of the new matrix
are the torus translates by the diagonal
action ${*}_A$ in~\eqref{eq:action} of the rows of the matrix $A^{(1)} (\mathbf 1) = A$.
 We can write this as ${\mb t} {*}_{A} A^{(1)} (\mb 1) = A^{(1)}(\mb t)$.
 Similarly, the row vectors of $A^{(k)}(\mathbf t)$ span the $k$th osculating space to $X_{A}$ at the point $\bf p$. Multiplying its rows by corresponding powers of $t_1, \dots, t_m$, we see that
the rowspan of  $A^{(k)}(\mathbf t)$  equals the rowspan of the matrix  ${\mb t} {*}_A {A^{(k)}} (\mb 1) = {\mb t} {*}_A A^{(k)}$.

 Introduce $\binom{m+k}k -(m+1)$ new (ordered) variables $u_{\mathbf i}$, where $\mathbf i =(i_1,\dots,i_m)$ and $2\le |\mathbf i|\le k$.
 The entry in the matrix $A^{(k)}$ in the $j$th column and the $\mathbf i$th row is the integer $\binom{\mathbf a_j}{\mathbf i}:=\binom{a_{1,j}}{i_1}\cdots \binom{a_{m,j}}{i_m}$.
Hence the toric variety $X_{A^{(k)}}$ has a parameterization
 \[\iota_{A^{(k)}}\colon
 (\mathbf t, \mathbf u):=(t_1,\dots,t_m, \dots,u_{\mathbf i},\dots) \mapsto (\cdots : \mathbf t^{\mathbf a_j}\prod_{\mathbf i}u_{\mathbf k}^{\binom{\mathbf a_j}{\mathbf i}} : \cdots).\]
  It follows that the row span of the matrix $A^{(k)}(\mathbf t,\mathbf u)$ is equal to the tangent space to $X_{A^{(k)}}$ 
  at the point $(\cdots : \mathbf t^{\mathbf a_j}\prod_{\mathbf i}u_{\mathbf i}^{\binom{\mathbf a_j}{\mathbf i}} : \cdots)$. This is a point of $X_A$ if all  $u_{\mathbf i}=1$.
   In particular, we have $X_A\subseteq X_{A^{(k)}}$. The tangent space to $X_{A^{(k)}}$ at the point ${\bf p} \in X_A$ is equal to 
     the torus translate by $({\bf t},{\bf 1})$ by means of the action ${*}_{A^{(k)}}$ of the embedded tangent space  at the point $(1: \dots: 1) \in  \mathbb P^d$, 
     which equals the $k$th osculating space to $X_A$ at $\bf p$.
    It follows that $X_{A^{(k)}}$ is a $k$th toric interpolant of $X_A$ at any point $\bf p$ in the torus of $X_A$.
  
   Assume a configuration $B$ defines a $k$th-toric interpolant of $X_A$. Then, we have the equality of tangent spaces $T_{B, \mb 1} = T_{A^{(k)}, \mb 1}$ since both coincide with the $k$th-osculating space to $X_A$ at $\mb 1$. Theorem~\ref{thm:unique} implies that
   $X_B = X_{A^{(k)}}$.
 \end{proof}
 
\begin{example}
Consider the Del Pezzo surface $X\subset \mathbb P^6$  of degree 6 given by the parameterization
\[\iota\colon (t_0,t_1,t_2)\mapsto (t_0:t_0t_1:t_0t_2:t_0t_1t_2:t_0t_1t_2^2:t_0t_1^2t_2:t_0t_1^2t_2^2).\]
We get
\[
A^{(2)}=\left(
\begin{array}{ccccccc} 
1&1&1 &1&1&1&1\\
0&1&0&1&1&2&2\\
0&0&1& 1&2&1&2\\
0&0&0&0&0&1&1\\
0&0&0&1&2&2&4\\
0&0&0&0&1&0&1
\end{array}\right).
\]
Consider new variables $u_{(2,0)}$, $u_{(1,1)}$ and $u_{(0,2)}$. Then we get a parameterization
\begin{multline*}
\iota_{A^{(2)}}\colon (t_0,t_1,t_2,u_{(2,0)},u_{(1,1)},u_{(0,2)})\mapsto (1:t_1:t_2:t_1t_2:t_1t_2u_{(1,1)}:\\
t_1t_2^2 u_{(2,0)}u_{(1,1)}^2:t_0t_1^2t_2u_{(1,1)}^2u_{(0,2)}:t_0t_1^2t_2^2u_{(2,1)}u_{(1,1)}^4u_{(0,2)}).
\end{multline*}
We could also subtract twice the sum of the fourth and sixth rows from the fifth row, which gives
the vector $(0,0,0,1,0,0,0)$, so that we get the parameterization:
\begin{multline*}
\iota_{A^{(2)}}\colon (t_0,t_1,t_2,u_{(2,0)},u_{(1,1)},u_{(0,2)})\mapsto \\(1:t_1:t_2:t_1t_2:t_1t_2u_{(1,1)}:t_1t_2^2 u_{(2,0)}:t_0t_1^2t_2u_{(0,2)}:t_0t_1^2t_2^2u_{(2,0)}u_{(0,2)}).
\end{multline*}
The corank of $A^{(2)}$ is $1$ and the $2$nd toric interpolant is the degree $3$ hypersurface with equation $x_0 x_4 x_5 - x_1 x_2 x_6 = 0$. This equation does
not depend on $x_3$ because the configuration of columns of $A^{(2)}$ is a pyramid
with vertex on its fourth column, as by the previous calculation the fourth column vector lies in the hyperplane $t_4 - 2 t_3 - 2 t_5=1$, while all the other column vectors lie in the parallel
hyperplane $t_4- 2 t_3 - 2 t_5=0$.

The tangent space to $X_{A^{(2)}}$ at a point 
\[\iota_{A^{(2)}}(\mathbf t,\mathbf u)=\iota_{A^{(2)}}(1,t_1,t_2,u_{(2,0)},u_{(1,1)},u_{(0,2)})\]
is the row span of the matrix $A^{(2)}(\mathbf t,\mathbf u)$ given by
\[
\left(
\begin{array}{ccccccc} 1&t_1&t_2&t_1t_2u_{(1,1)}&t_1t_2^2u_{(1,1)}^2u_{(0,2)}&t_1^2t_2u_{(2,0)}u_{(1,1)}^2&t_1^2t_2^2u_{(2,0)}u_{(1,1)}^4u_{(0,2)}\\
0&1&0&t_2u_{(1,1)}&t_2^2u_{(1,1)}^2u_{(0,2)}&2t_1t_2u_{(2,0)}u_{(1,1)}^2&2t_1t_2^2u_{(2,0)}u_{(1,1)}^4u_{(0,2)}\\
0&0&1& t_1u_{(1,1)}&2t_1t_2u_{(1,1)}^2u_{(0,2)}&t_1^2u_{(2,0)}u_{(1,1)}^2&2t_1^2t_2u_{(2,0)}u_{(1,1)}^4u_{(0,2)}\\
0&0&0&0&0&t_1^2t_2u_{(1,1)}^2&t_1^2t_2^2u_{(1,1)}^4u_{(0,2)}\\
0&0&0&t_1t_2&2t_1t_2^2u_{(1,1)}u_{(0,2)}&2t_1^2t_2u_{(2,0)}u_{(1,1)}&4t_1^2t_2^2u_{(2,0)}u_{(1,1)}^3u_{(0,2)}\\
0&0&0&0&t_1t_2^2u_{(1,1)}^2&0&t_1^2t_2^2u_{(2,0)}u_{(1,1)}^4
\end{array}\right)
\]
For a point $\iota_{A^{(2)}}(t_1,t_2,1,1,1)\in X$ with $t_1,t_2\neq 0$, this is the same as the row span of the matrix
\[
\left(
\begin{array}{ccccccc} 1&t_1&t_2&t_1t_2&t_1t_2^2&t_1^2t_2&t_1^2t_2^2\\
0&1&0&t_2&t_2^2&2t_1t_2&2t_1t_2\\
0&0&1& t_1&2t_1t_2&t_1^2&2t_1^2t_2\\
0&0&0&0&0&1&t_2\\
0&0&0&1&2t_2&2t_1&4t_1t_2\\
0&0&0&0&1&0&t_1
\end{array}\right).
\]
which is the same as the second osculating space to $X$ at this point.
\end{example}

 \section{Toric curves}\label{sec:curves}

 We shall now study in more detail the case when $X$ is a curve. A toric curve $X\subset  \mathbb P^d$ which is not contained in a coordinate hyperplane, can be described as follows.
Given integers $\ell_0=0<\ell_1< \cdots < \ell_{d-1}<l_d$, set $\ell:=\{\ell_0,\dots,\ell_d\}$ and
\[A_{\ell,d}:= \left( \begin{array}{cccc}
1&1&\cdots &1\\
0&\ell_1&\cdots &\ell_d\\
\end{array}\right).\]
Let $X_{A_{\ell,d}}\subset \mathbb P^d$ be the rational curve of degree $\ell_d$ parameterized by 
\[t\mapsto (1:t^{\ell_1}:\cdots:t^{\ell_{d-1}}:t^{\ell_d}).\]
In the special case when $\ell_0=1$, $\ell_1=1$,\dots, $\ell_d=d$, the curve $X_{A_{\ell,d}}$ is a \emph{rational normal curve}.

The $k$th toric interpolant of $X_{A_{\ell,d}}$ is given by the matrix
\[ A_{\ell,d}^{(k)}= \left( \begin{array}{cccccc}
1&1&1&\cdots&1&1 \\
0 &\ell_1&\ell_2&\cdots&\ell_{d-1}&\ell_d\\
0&\binom{\ell_1}2&\binom{\ell_2}2&\cdots&\binom{\ell_{d-1}}2&\binom{\ell_d}2\\
\vdots&\vdots&\vdots&\cdots&\vdots&\vdots\\
0&\binom{\ell_1}k&\binom{\ell_2}k&\cdots &\binom{\ell_{d-1}}k&\binom{l\ell_d}k\\
\end{array}\right).\]

As we observed in the comment after Definition \ref{def:Ak} there exists a $(k+1)\times (k+1)$-matrix $M_k$ such that the matrix
\[ \widetilde A_{\ell,d}^{(k)}:= \left( \begin{array}{cccccc}
1&1&1&\cdots&1&1 \\
0 &\ell_1&\ell_2&\cdots&\ell_{d-1}&\ell_d\\
0&\ell_1^2&\ell_2^2&\cdots&\ell_{d-1}^2&{\ell_d}^2\\
\vdots&\vdots&\vdots&\cdots&\vdots&\vdots\\
0&\ell_1^k&\ell_3^k&\cdots &\ell_{d-1}^k&\ell_d^k\\
\end{array}\right)\]
is equal to $M_k \cdot A_{\ell,d}^{(k)}$, with $\det M_k=\prod_{j=1}^k j!$.
The column vectors of this second matrix $\widetilde A_{\ell,d}^{(k)}$ form the vertices of a cyclic polytope $C_{\ell,d}^{(k)}$.
They are ordered points on a moment curve.  Here, ordered means naively that the second coordinates are increasing, but the main point is that all the maximal minors of the matrix (with the corresponding ordered columns)  are positive.  Therefore, we can compute the degree of the variety $X_{A_{\ell,d}}^{(k)} = X_{\widetilde A_{\ell,d}}^{(k)}$
as the normalized lattice volume of the cyclic polytope $C_{\ell,d}^{(k)}$ with respect
to the lattice generated by the columns of the matrix $A_{\ell,d}^{(k)}$, which are all consecutive vertices of the cyclic polytope. 
Note that the matrices $\widetilde A_{\ell,d}^{(k)}$ are {\em positroids}:  their maximal minors are positive since they are equal to Vandermonde determinants and the $\ell_i$ satisfy $\ell_i < \ell_j$ for $i < j$. Since $A_{l,d}^{(k)}=M_k^{-1} \cdot \widetilde A_{l,d}^{(k)}$, with $\det M_k^{-1}>0$, also $A_{\ell,d}^{(k)}$ is a positroid. 

In general, the {\em lattice volume} of a $k$-dimensional lattice polytope is $k!$ times its Euclidean volume and will be denoted by $\Vol$.
By decomposing the polytope into simplices, $\Vol(C_{\ell,d}^{(k)})$ can be computed as the sum of the volumes of these simplices, which will
 all be expressed as Vandermonde determinants in $\ell_1,\dots,\ell_d$. To decompose the polytope, one can use a classical characterization 
 of its facets, for an account of this, see \cite{EFP}.
Once one knows the volume, one can compute the number of lattice points of the polytope by
the formula given by 
Liu \cite{Liu}*{Thm.~1.2, p.~112}:
\[i(C_{\ell,d}^{(k)})=\sum_{j=0}^k \Vol C_{\ell,d}^{(j)},\]
where $\Vol C_{\ell,d}^{(0)}:=1$. 

To illustrate this situation, assume that $k=2$. Then $C_{\ell,d}^{(2)}$ is a cyclic polygon.

\begin{prop}\label{curves}
The lattice volume of and the number of lattice points in $C_{\ell,d}^{(2)}$ are given as
\[\Vol C_{\ell,d}^{(2)}=\sum_{i=1}^{d-1} \ell_i \ell_{i+1}(\ell_{i+1}-\ell_i).\]
and
\[i(C_{\ell,d}^{(2)})=\frac{1}2 \sum_{i=1}^{d-1} \ell_i \ell_{i+1}(\ell_{i+1}-\ell_i)+\ell_d +1.\]
\end{prop}

\begin{proof}
We write $C_{\ell,d}^{(2)}$  as the union of the triangles with vertices $(0,0)$, $(\ell_i,\ell_i^2)$,and $(\ell_{i+1},\ell_{i+1}^2)$.
The lattice area of the triangles are $\frac{1}2 \ell_i \ell_{i+1}(\ell_{i+1}-\ell_i)$. The lattice length of 
$C_{\ell,d}^{(1)}$ is $\ell_d$.
\end{proof}

Assume  $\ell_0=0, \ell_1=1,\ldots, \ell_d=d$ and $k=2$.
If we instead use the matrix 
\begin{equation} \label{ad2}
A_d^{(2)}= \left( \begin{array}{cccccccccc}
1&1&1&1&1&\cdots&1&1&1&1 \\
0 &1&2&3 &4&\cdots&d-3&d-2&d-1&d\\
0&0&1&3&6&\cdots&\binom{d-3}2&\binom{d-2}2&\binom{d-1}2&\binom{d}2\\
\end{array}\right),
\end{equation}
 we get a different polygon, which we call $P(d)$. It is easy to see that the lattice area of $P(d)$ coincides with its normalized area and thus gives the degree of the associated
 toric variety $X_{A_d^{(2)} } = X_{\widetilde A_{\ell,d}^{(2)}}$.

 \begin{figure}\label{P(d)}
  \includegraphics[angle=0,width=0.4\textwidth]{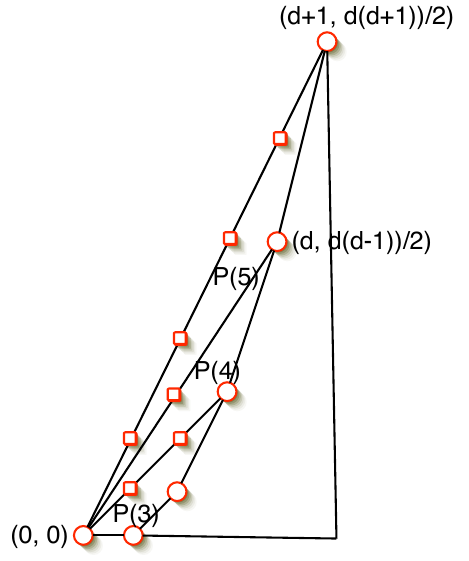}
    \caption{The polygons $P(d)$.}
    \end{figure}

The $(d-1)$th toric interpolant of $X_{A_d}$ is given by the matrix
\[ A_d^{(d-1)}= \left( \begin{array}{cccccccccc}
1&1&1&1&1&\cdots&1&1&1&1 \\
0 &1&2&3 &4&\cdots&d-3&d-2&d-1&d\\
0&0&1&3&6&\cdots&\binom{d-3}2&\binom{d-2}2&\binom{d-1}2&\binom{d}2\\
\vdots&\vdots&\vdots&\vdots&\vdots&\cdots&\vdots&\vdots&\vdots&\vdots\\
0&0&0&0&0&\cdots&0&0&1&\binom{d}{d-1}\\
\end{array}\right).\]
Hence we get the following parameterization of $X_{A_d^{(d-1)}}$:
\[(t_0,t_1,\dots,t_{d-1})\mapsto (t_0:t_0t_1:t_0t_1^2t_2:\cdots:t_0t_1^dt_2^{\binom{d}2}\cdots t_{d-1}^d)\in \mathbb P^d.\]

\begin{prop} \label{prop:d-1}
 $X_{A_d^{(d-1)}}$ is a toric hypersurface of degree $2^{d-1}$. 
\end{prop}

\begin{proof}
It is clear that the rank of $A_d^{(d-1)}$ is $d$ as its  maximal minor corresponding to its first columns is equal to $1$. Then, its kernel has dimension one and so  $X_{A_d^{(d-1)}}$ is a hypersurface.
It is easy to check that a $(d+1)$-vector $v$ lies in the kernel of $A_d^{(d-1)}$ if and only if the polynomial $f_v(t) = \sum_{i=0}^d v_i t^i$ vanishes at $t=1$ jointly with its
derivatives up to order $(d-1)$. Then the  vector $w$ with coordinates $w_i = (-1)^{d-i} \binom{d}{i}$, satisfies that $f_w= (x-1)^d$, and so $w$ is a generator of the kernel.
Separating its positive from its negative entries, we get that the ideal of $X_{A_d^{(d-1)}}$ is generated by the binomial
\[x_0x_2^{\binom{d}2}\cdots x_{d-2}^{\binom{d}{d-2}}x_d - x_1^{\binom{d}1}x_3^{\binom{d}3} \cdots x_{d-1}^{\binom{d}{d-1}}=0 \]
if $d$ is even,
and 
\[ x_0x_2^{\binom{d}2}\cdots x_{d-1}^{\binom{d}{d-1}} - x_1^{\binom{d}1}x_3^{\binom{d}3} \cdots x_{d-2}^{\binom{d}{d-2}}x_d=0 \]
if $d$ is odd.  By using $\binom{d}i=\binom{d-1}{i-1}+\binom{d-1}i$, we see that all four monomials appearing in these two equations have degree $\sum_{i=0}^{d-1}\binom{d-1}i=2^{d-1}$.
\end{proof}

In \cite{PosGeo}, the concept of \emph{positive geometries} was introduced in the study of scattering amplitudes in particle physics. A positive geometry is a real semi-algebraic set, together with a rational differential form, with poles along the boundary of the semi-algebraic set, called the \emph{canonical form}. Of special interest were semi-algebraic sets given by certain polytopes, called \emph{generalized  tree amplituhedra}, of type $\mathcal A_{n,k,m}(Z)$, where $Z$ is a $(k+m)\times n$-matrix with positive maximal minors -- a positroid -- in the real
positive Grassmann variety $\mathbb G^{\ge 0}_{k,k+m}$ \cite{ampli}.
The polygon $P(d)$ is \emph{cyclic} and is an example of a generalized  tree amplituhedron of type 
$\mathcal A_{d+1,1,2}(Z)$ in  $\mathbb G^{\ge 0}_{1,3}=(\mathbb P^2)^{\ge 0}$, with $Z=A_d^{(2)}$.

Any convex polygon gives rise to a 
 positive geometry, hence has an associated canonical form. In particular, $P(d)$ gives rise to a positive geometry.
The following example shows an explicit computation for any $d$ of the procedure outlined in~\cite{PosGeo}*{\S~7.1.1} for computing the canonical form of a cyclic polytope.
 Using the description of the facets of combinatorial cyclic polytopes~\cite{CD}, our computation could be extended to higher interpolants of toric curves.

\begin{example}\label{prop:pos}

The polygon $P(d)$ can be decomposed as the polygon $P(d-1)$ union the triangle $\Delta(d)$ with vertices $(0,0)$, $(d-1,\binom{d-1}2)$, and $(d,\binom{d}2)$.
The sides of $\Delta(d)$ have equations $2y-(d-2)x=0$, $2y-2(d-1)x+d(d-1)=0$, and $(d-1)x-2y=0$. Hence we compute the
canonical form (see \cites{PosGeo,adj}) of the triangle to be
\[\Omega(\Delta(d))=\frac{2d(d-1)}{(2y-(d-2)x)(2y-2(d-1)x+d(d-1))((d-1)x-2y)} dx\wedge dy.
\]
The additivity of the canonical form \cite{adj}*{2.4} then gives
\[\Omega(P(d))=\frac{2d(d-1)\, dx\wedge dy}{(2y-(d-2)x)(2y-2(d-1)x+d(d-1))((d-1)x-2y)} + \Omega(P(d-1))\]
Hence the canonical form of the polygon $P(d)$ equals
\[\Omega(P(d))=\sum_{i=2}^d \frac{2i(i-1)}{(2y-(i-2)x)(2y-2(i-1)x+i(i-1))((i-1)x-2y)} dx\wedge dy.\]

In particular, we get
\[\Omega(P(2))=\frac{1}{y(y-x+1)(x-2y)} dx\wedge dy\]
and
\[ \Omega(P(3))=\frac{2y-2x+3}{y(y-x+1)(y-2x+3)(x-y)} dx\wedge dy.\]
The latter is consistent with the fact that the adjoint curve to $P(3)$ is the line $2y-2x+3=0$ \cite{adj}*{2.1}.
\end{example}

\section{Normalizations and dual varieties} \label{sec:dual}
In projective geometry, it is interesting to determine whether a given linear projection is \emph{general}. Indeed, a character (a cycle class, or a number) of a projective variety is said to be \emph{projective} if it is invariant under a \emph{general} linear projection. 
An example of a projective character of a variety is its degree, and also the degree of its dual variety, i.e., the variety in the dual projective space equal to the (closure of) the set of tangent hyperplanes. 

 Given a lattice point configuration $A\subset \mathbb Z^{m+1}$, let $\overline{A}\subset \mathbb Z^{m+1}$ 
 denote the configuration of lattice points in the convex hull of $A$ in $\R^{m+1}$,     
 and denote by $\overline A$ 
 the corresponding matrix. We obtain a \emph{normal} toric variety $X_{\overline A} \subset \mathbb P^N$, 
 where $N+1$ is the number of lattice points in the configuration $\overline{A}$. Since the matrix $A$ is obtained from $\overline A$ by removing some columns, $X_A$ is equal to a linear projection of its normalization $X_{\overline A}$.

We now return to the second toric interpolant $X_{A_d^{(2)}}$ of the rational normal curve $X_A$ of degree $d$ studied in the previous section, see (\ref{ad2}). We assume $d\ge 3$.
We choose this case as an example since it is computable, and we show that 
$X_{A_d^{(2)}}$ is \emph{not} a general linear projection of its normalization $X_{\overline{A_d^{(2)}}}$. For $d\ge 4$,  this follows since the degrees of their dual varieties are different. For $d=3$, the surface $X_{A_3^{(2)}}\subset \mathbb P^3$ has a singular line of multiplicity 3, hence is not a general projection of a smooth surface in $\mathbb P^5$.

\begin{prop}
The degree of $X_{A_d^{(2)}}\subset \mathbb P^d$ is $\binom{d+1}3$, and it is a toric linear projection of its normalization 
$X_{\overline{A_d^{(2)}}}$ from $\mathbb P^{\frac{1}{12}d(d^2+8)}$ if $d$ is even and from $\mathbb P^{\frac{1}{12}d(d^2+11)}$ if $d$ is odd.
\end{prop}

\begin{proof}
The lattice polygon $P(d)$ in the plane corresponding to $A_d^{(2)}$ contains the lattice polygon $P(d-1)$ corresponding to $A_{d-1}^{(2)}$, see Fig.~1.
It is easy to see that the difference of the two polygons has lattice area equal to 
\[d\binom{d}2 -(d-1)\binom{d-1}2 - 2\binom{d-1}2 -\binom{d}2+\binom{d-1}2=\binom{d}2.\] 
Hence we get
\[ \deg X_{A_d^{(2)}} = \deg X_{A_{(d-1)}^{(2)}} + \binom{d}2=\cdots = \sum_{j=2}^d \binom{j}2=\binom{d+1}3.\]

Let $p(d)$ denote the perimeter of $P(d)$. This polygon has $d +1$ vertices and $d +1$ edges. Of the latter, $d$  have lattice length 1. 
The edge between $(d,\binom{d}2)$ and $(0,0)$ is a segment of the line $y=\frac{d-1}2 x$, which contains $\frac{1}2(d-2)$ lattice points 
other than the vertices if $d$ is even and $d-1$ if $d$ is  odd. Hence $p(d)=d+\frac{1}2(d-2)+1=\frac{3}2 d$ if $d$ is even and $p(d)=d+d-1+1=2d$ if $d$ is odd.
It then follows from Pick's formula that the number of lattice points in $P(d)$ is $\frac{1}2 \binom{d+1}3 +\frac{1}2\cdot \frac{3}2 d +1
=\frac{1}{12}d(d^2+8)+1$ if $d$ is even and $\frac{1}2 \binom{d+1}3 +\frac{1}2 \cdot 2d+1=\frac{1}{12}d(d^2+11)+1$ if $d$ is odd.
\end{proof}

\bigskip

\begin{prop}\label{rn}
The surface $X_{\overline{A_3^{(2)}}}$
is a quartic nonsingular surface in $\mathbb P^5$.
For $d\ge 4$, the surface $X_{\overline{A_d^{(2)}}}$ has two singular points, both of multiplicity $d-1$ if $d$ is even and  of multiplicity $\frac{1}2 (d-1)$ if $d$ is odd.
\end{prop}

\begin{proof}
The vertices of the polygon $P(d)$ are $(0,0), (1,0),\cdots, (d-1,\binom{d-1}2), (d,\binom{d}2)$. 
It is easy to check that the vertices $(1,0),\cdots, (d-1,\binom{d-1}2)$ of $P(d)$ are nonsingular, so it remains to check the vertices $(0,0)$ and $(d,\binom{d}2)$ \cite{GKZ}*{Thm.~3.14, p.~186}: 

\begin{itemize}
\item If $d$ is even:
\[
m_{(0,0)}=|\det\left(
\begin{array}{cc} 1&2\\
0&d-1
\end{array}\right)|=d-1, \]
\[
m_{(d, \binom{d}{2})}=|\det\left(\begin{array}{cc}
2&1\\
d-1&d-1
\end{array}\right)|=d-1.
\]
 \item If $d$ is odd:
\[
m_{(0,0)}=|\det\left(
\begin{array}{cc} 1&1\\
0&\frac{1}2(d-1)
\end{array}\right)|=\frac{1}{2}(d-1), \]
\[
m_{(d,
\binom{d}{2})}=|\det\left(\begin{array}{cc} 1&1\\
\frac{1}{2}(d-1)&d-1
\end{array}\right)|=\frac{1}{2}(d-1).
\]
 \end{itemize}
 This proves the proposition.
\end{proof}

Finally, we derive formulas for the degrees of the dual varieties, which enable us to determine whether the linear projection is generic. 

\begin{prop}
The degree of the dual variety $(X_{\overline{A_d^{(2)}}})^\vee$ is equal to $\binom{d-1}2(d+3)$
if $d$ is even and $\frac{1}2(d-1)(d^2+d-8)$ if $d$ is odd.

The degree of the dual variety $(X_{A_d^{(2)}})^\vee$ is equal to 
$\binom{d-1}2(d+1)$.
\end{prop}

\begin{proof}
We shall apply the formula for the degree of the dual variety of a toric surface (see \cite{MaTa}*{Cor.~1.6, p.~2042}, \cite{Local}*{Section 5}).
In the first case the surface is normal, and we get
 \[ \deg (X_{\overline{A_d^{(2)}}})^\vee = 3 \Vol(P(d)) - 2 p(d) +\sum_v \Eu(v),\]
where $\Eu(v)$ denotes the local Euler obstruction at the point of $X_A$ corresponding to the vertex $v$ of $P(d)$. The local Euler obstruction of a variety is a constructible function, introduced by MacPherson in order to define Chern classes of singular varieties \cite{euler}.

For a normal toric surface, 
the local Euler obstruction at a vertex $v$ is equal to  2 minus the difference between the lattice area of $P(d)$ and that of the polytope 
$\Conv(P(d)\setminus \{v\})$ obtained by removing the vertex $v$ and taking the convex hull of all the remaining vertices \cite{MaTa}*{Cor.~4.4, p.~2052}. It can also be 
computed as $1-c$, where $c$ is the number if interior lattice points of $P(d)$ which are boundary points of $\Conv(P(d)\setminus \{v\})$ \cite{Local}*{Lemma~5.1, p.~514}.
We get
\[\Eu ((0,0))=\Eu((d,\binom{d}2))=2-\frac{1}2d\]
if $d$ is even, and
\[\Eu ((0,0))=\Eu((d,\binom{d}2))=2-\frac{1}2 (d-1)\]
if $d$ is odd. Hence we get
\[\deg (X_{\overline{A_d^{(2)}}})^\vee =
3\binom{d+1}3-2\cdot \frac{3}2 d +d-1+2(2-\frac{1}2d)=\frac{1}2(d^3-7d+6)\]
if $d$ is even, and
\[ \deg (X_{\overline{A_d^{(2)}}})^\vee =3\binom{d+1}3-2\cdot 2d +d-1+2(2-\frac{1}2(d-1))=\frac{1}2 (d^3-9d+8)\]
if $d$ is odd.

In the second case, the surface is singular along the curve corresponding to the edge $\Delta$ of $P(d)$ joining the vertices $v:=(0,0)$ and $v':=(d,\binom{d}2)$. 
In the formula for the degree of the dual surface, also the edge lengths need to be weighted by the Euler obstruction at a general point of the corresponding orbit. 
This gives  \cite{MaTa}*{Cor.~1.6, p.~2042} $\Eu(\Delta)=\frac{1}2 d$ if $d$ is even and $d$ if $d$ is odd. The Euler obstruction at each singular vertex 
is then $\Eu(v)=\Eu(v')=\frac{1}2 d+1-\binom{d}2$ if $d$ is even and $\Eu(v)=\Eu(v')=d+1-\binom{d}2$ if $d$ is odd. We get
\[\deg (X_{A_d^{(2)}})^\vee =
3\binom{d+1}3-2(\frac{1}2 d +d)+d-1+2(\frac{1}2d+1-\binom{d}2)=\frac{1}2(d^3-2d^2-d+2)\]
if $d$ is even, and
\[ \deg (X_{A_d^{(2)}})^\vee =3\binom{d+1}3-2(d+d) +d-1+2(d+1-\binom{d}2)=\frac{1}2 (d^3-2d^2-d+2)\]
if $d$ is odd.
\end{proof}

\begin{cor}
The surface $X_{A_d^{(2)}}$ is not a general linear projection of its normalization $X_{\overline{A_d^{(2)}}}$.
\end{cor}

\begin{proof}
We have $\deg (X_{A_d^{(2)}})^\vee \ne \deg (X_{\overline{A_d^{(2)}}})^\vee$ for $d\ge 4$.
For $d=3$, we have
 $\deg (X_{\overline{A_3^{(2)}}})^\vee=\deg (X_{A_3^{(2)}})^\vee=4$.
The nonsingular quartic surface $X_{\overline{A_3^{(2)}}}\subset \mathbb P^5$ is given by the polygon $P(3)$. 
Its projection $X_{A_3^{(2)}}\subset \mathbb P^3$ has a triple line, corresponding to the edge $\Delta$; indeed, its equation is $x_0x_2^3-x_1^3x_3=0$. Since general projections to $\mathbb P^3$ of smooth surfaces in $\mathbb P^5$ do not have singular curves of multiplicity $> 2$, $X_{A_3^{(2)}}$ is not a general projection of $X_{\overline{A_d^{(2)}}}$, even though their dual varieties have the same degree.
\end{proof}

The surfaces $X_{\overline{A_4^{(2)}}}$  and $X_{\overline{A_5^{(2)}}}$ have two singular points with
 local Euler obstruction equal to 0, hence are Gorenstein varieties, whereas for $d\ge 6$,  $X_{\overline{A_d^{(2)}}}$ is not Gorenstein. 
 The surfaces $X_{A_d^{(2)}}$ are non-normal (they have a singular curve corresponding to the edge $\Delta$).

\end{document}